\documentclass[12pt]{amsart}
\usepackage{amsmath,amssymb, euscript, graphicx, hyperref}

\newtheorem{thm}{Theorem}[section]
\newtheorem{lem}[thm]{Lemma}
\newtheorem{prop}[thm]{Proposition}
\theoremstyle{definition}
\newtheorem{defn}[thm]{Definition}

\newtheorem{cor}[thm]{Corollary}
\newtheorem{question}[thm]{Question}

\newcommand{\Diff}{\mathop{\rm Diff}}
\newcommand{\Homeo}{\mathop{\rm Homeo}}

\newcommand{\supp}{\mathop{\rm supp}}

\begin{document}

\author{Azer Akhmedov}
\address{Azer Akhmedov,
Department of Mathematics,
North Dakota State University,
PO Box 6050,
Fargo, ND, 58108-6050}
\email{azer.akhmedov@ndsu.edu}

\author{Michael P. Cohen}
\address{Michael P. Cohen,
Department of Mathematics,
North Dakota State University,
PO Box 6050,
Fargo, ND, 58108-6050}
\email{michael.cohen@ndsu.edu}

\title[Existence and genericity of topological generating sets]{Existence and genericity of finite topological generating sets for homeomorphism groups}

\begin{abstract}  We show that the topological groups $\Diff_+^1(I)$ and $\Diff_+^1(\mathbb{S}^1)$ of orientation-preserving $C^1$-diffeomorphisms of the interval and the circle, respectively, admit finitely generated dense subgroups.  We also investigate the question of genericity (in the sense of Baire category) of such finite topological generating sets in related groups.  We show that the generic pair of elements in the homeomorphism group $\Homeo_+(I)$ generate a dense subgroup of $\Homeo_+(I)$.  By contrast, if $M$ is any compact connected manifold with boundary other than the interval, we observe that an open dense set of pairs from the associated boundary-fixing homeomorphism group $\Homeo_0(M,\partial M)$ will generate a discrete subgroup.  We make similar observations for homeomorphism groups of manifolds without boundary including $\mathbb{S}^1$.
\end{abstract}

\maketitle


\section{Introduction}

A finite collection of elements $g_1,g_2,...,g_n$ in a separable topological group $G$ is called a (finite) \textit{topological generating set} for $G$ if the countable group $\Gamma=\langle g_1,g_2,...,g_n\rangle$ which they generate is dense in $G$.  In this case we say $G$ is \textit{topologically $n$-generated}.  In 1990, Hofmann and Morris \cite{HM} showed that every separable compact connected group is topologically $2$-generated.  Their result built upon earlier work of Kuranishi \cite{K} for compact semisimple Lie groups.  Very recently, Gelander and Le Maitre \cite{GL} have applied the solution to Hilbert's Fifth Problem to give a very general result: every separable connected locally compact group is topologically finitely generated.

\medskip

The homeomorphism and diffeomorphism groups of compact manifolds are far from locally compact, and so different tools are needed for their study.  We are partially motivated in this paper by a desire to study these groups, loosely, as infinite-dimensional analogues of Lie groups.  In particular we seek to understand their discrete and dense finitely generated subgroups.  The reader may consult the paper \cite{A} of the first named author for a large number of remarks and open questions in this program.

\medskip

It is very well-known that the group $\Homeo_+(I)$ of orientation-preserving homeomorphisms of the interval $I=[0,1]$ is topologically $2$-generated.  The most popular example of a $2$-generated dense subgroup in $\Homeo_+(I)$ is Thompson's group $F$, represented as a group of piecewise linear homeomorphisms.  It is also well-known that the usual embedding of $F$ may be ``smoothed out'' to find representations of $F$ which lie in $\mathrm{Diff}_+^k(I)$, for all $1\leq k\leq \infty$ \cite{GS}.  However, these representations of $F$ are never dense in $\mathrm{Diff}_+^k(I)$ with respect to the appropriate $C^k$ topology.  In fact, no finitely generated group has been shown to admit a $C^1$-dense representation in $\mathrm{Diff}_+^1(I)$.  Given that many separable topological groups do not admit finitely generated dense subgroups at all (this is easy to see in the additive product group $\mathbb{Z}^\mathbb{N}$ for instance), we are motivated by the question: is $\Diff_+^1(I)$ topologically finitely generated or not? This question is especially intriguing since the topological group $\Diff_+^1(I)$ is homeomorphic to an infinite dimensional separable Banach space.{\footnote {Any two infinite dimensional separable Banach spaces are homeomorphic by a result of M.I.Kadets \cite{Ka}, and it is easy to see that the map $f\to \ln f'(t)- \ln f'(0)$ establishes a homeomorphism from $\Diff_+^1(I)$ to the Banach space $C_0[0,1]$.}}

\medskip

The first main result of our paper is to show that, despite a lack of simple examples, dense finitely generated subgroups of $\Diff_+^1(I)$ do exist.

\begin{thm} \label{thm:main} There exists a finitely generated dense subgroup of $\Diff_+^1(I)$.
\end{thm}

The core of the proof lies in constructing a set of seven topological generators $f,g,u,v,h,\phi,\psi$ for the commutator subgroup $G$ of $\mathrm{Diff}_+^1(I)$.  Stated informally, we use four generators $f,g,u,v$ to encode small copies of a countable dense subset $(\eta_n)_{n\geq 1}$ of $G$ along a sequence of disjoint subintervals $(I_n)_{n\geq 1}$ of $I$.  The fifth generator $h$ is used to ``erase'' finitely many excess copies when they are not needed, while the sixth and seventh generators $\phi$ and $\psi$ are used to magnify the small copies to approximate the maps $\eta_n$ to arbitrary precision.

\medskip

Our construction relies on the perfectness results of Tsuboi \cite{T} for diffeomorphism groups of the interval, as well as the following original lemma on approximation by diffeomorphisms with iterative $n$-th roots which is independently interesting.

\begin{lem}\label{lem:introroot} Let $f\in\Diff_+^1(I)$ without a fixed point in (0,1) and assume that $f'(0)=f'(1)=1$.  Then for every $r>0$, there exists a $g\in\Diff_+^1(I)$ and a positive integer $N$ such that $g$ is $r$-close to identity and $g^N$ is $r$-close to $f$, in the $C^1$ metric.
\end{lem}

Since $\Diff_+^1(I)$ admits a quotient group topologically isomorphic to $\mathbb{R}^2$, it is not too hard to see that a dense subgroup of $\Diff_+^k(I)$ requires at least $3$ generators.  Our construction uses $10$ generators, and while this number should be easy to reduce, we are not sure of the least number of generators necessary.  The least number of topological generators for a topological group $G$ is sometimes called the \textit{topological rank} of $G$ (see \cite{GL}).

\begin{question}  What is the topological rank of $\Diff_+^1(I)$?  Of $\Diff_+^k(I)$ for $1<k\leq\infty$?
\end{question}

The next corollary is immediate from our main theorem, using a standard fragmentation argument.

\begin{cor}  There exists a finitely generated dense subgroup of $\Diff_+^1(\mathbb{S}^1)$.
\end{cor}

\begin{proof}  Write $\Diff_+^1(\mathbb{S}^1)=\Diff_+^1([a,b])\cdot\Diff_+^1([c,d])$ where $(a,b)$ and $(c,d)$ are subarcs of the circle whose union is the entire circle.  $\Diff_+^1([a,b])$ and $\Diff_+^1([c,d])$ are both topologically isomorphic to $\Diff_+^1(I)$, so they have finite topological generating sets $F_1$ and $F_2$ respectively by Theorem \ref{thm:main}.  Then $F_1\cup F_2$ is a finite topological generating set for $\Diff_+^1(\mathbb{S}^1)$.
\end{proof}

In \cite{GL}, the authors discuss the concept of \textit{local rank}.  The local rank of a separable topological group is the minimal number $n$ such that, within any neighborhood of identity, there exists a set of $n$ many topological generators for the group.  If no such $n$ exists, the local rank of the group is said to be infinite.  Examination of the proof of Theorem \ref{thm:main} reveals that the generators we construct may be chosen arbitrarily close to identity in the $C^1$ metric, so the local rank of $\Diff_+^1(I)$ is at most $10$.  In particular, $\Diff_+^1(I)$ and $\Diff_+^1(\mathbb{S}^1)$ have finite local rank.

 

 

In the second part of this paper, we are concerned with the question of the generic behavior of finite topological generating sets.  The question dates back to Schreier and Ulam (\cite{SU}), who showed that for every compact connected metrizable group $G$, the set of pairs $(g_1,g_2)\in G$ which generate a dense subgroup of $G$ has full Haar measure in $G^2$.  More recently, Winkelmann \cite{W} has proved the following striking result in the case of connected Lie groups $G$: if $G$ is amenable then there exists an $N$ such that the set of $n$-tuples $(g_1,...,g_n)$ which generate a dense subgroup has zero Haar measure for all $n<N$, while the set of $n$-tuples $(g_1,...,g_n)$ which generate a discrete subgroup has zero Haar measure for all $n\geq N$.  If $G$ is non-amenable, then the set of discrete finitely generated groups and the set of non-discrete finitely generated groups both have infinite measure, for every $n > 1$.

\medskip

In this paper we are primarily concerned with groups which are not locally compact and thus do not carry Haar measures.  As the natural alternative, we approach the question in terms of Baire category.  Recall that each group of the form $\Homeo_0(M)$ or $\Diff_0^k(M)$ for a compact manifold $M$ is a Polish (i.e. separable completely metrizable) topological group and thus satisfies the hypothesis of the Baire category theorem (here, $\Homeo_0(M)$ denotes the connected component of identity in the full homeomorphism group $\Homeo(M)$; similarly, $\Diff_0^k(M)$ denotes the connected component of identity in the full diffeomorphism group $\Diff^k(M)$).  We will say that a Polish topological group $G$ is \textit{generically topologically $n$-generated} if the set\\

\begin{center} $\Omega_n=\{(g_1,g_2,...,g_n)\in G^n:\langle g_1,g_2,...,g_n\rangle$ is dense in $G\}$
\end{center}
\vspace{.3cm}

\noindent is comeager in $G^n$ (i.e. if $\Omega_n$ contains a countable intersection of dense open subsets of $G^n$).  If a group is generically topologically $n$-generated, then by the Kuratowski-Ulam theorem, it is generically topologically $m$-generated for all $m\geq n$.

\medskip

As an example, each torus $\mathbb{T}^n=(\mathbb{R}/\mathbb{Z})^n$ is topologically $1$-generated, so $\Omega_1$ is non-empty.  This is a consequence of a classical theorem of Leopold Kronecker (\cite{Kr}), which asserts that $\Omega_1$ of $\mathbb{T}^n$ consists exactly of those tuples $(g_1,...,g_n)\in\mathbb{T}^n$ for which $\{1,g_1,...,g_n\}$ is a linearly independent set over the rationals.  But this characterization implies something stronger, that $\Omega_1$ is in fact comeager in $\mathbb{T}^n$.  Similarly, one may apply Kronecker's theorem to show not only that each additive group $\mathbb{R}^n$ is $(n+1)$-generated, but moreover that $\Omega_{n+1}$ is comeager in $\mathbb{R}^n$.  By contrast, Kechris and Rosendal (\cite{KR}) showed that the infinite permutation group $S_\infty$, endowed with its usual topology, is topologically $2$-generated, but since $S_\infty$ is a totally disconnected group, $\Omega_n$ cannot be dense in $S_\infty$, let alone comeager.  Rather, by direct argument it is easy to see that $\Omega_n$ is nowhere dense in $S_\infty$ for all $n$.

\medskip

Our main result on this topic is the following.

\begin{thm} \label{thm:generic}  $\Homeo_+(I)$ is generically topologically $2$-generated.
\end{thm}

It is striking that finitely generated dense subgroups are ubiquitous at the $C^0$ level in $\Homeo_+(I)$, and yet, speaking empirically, to find even one $C^1$-dense finitely generated subgroup in $\Diff_+^1(I)$ presents a great challenge.  We remark that $\Homeo_+(I)$, despite being non-locally compact, is in some ways analogous to a compact connected topological group, in the sense that it always has finite diameter with respect to any left-invariant metric on the group{\footnote {See \cite{DH}, as well as \cite{R} for exposition on the consequences of this property.}}, and thus has bounded large scale metric geometry.  The diffeomorphism group $\Diff_+^1(I)$, by contrast, has a very complicated and certainly unbounded large scale geometry.  This fact may be reflected in the seeming paucity of finitely generated dense subgroups there.

\medskip

In this paper we also point out that $I$ is the unique compact connected manifold with boundary in which generic finitely generated groups of homeomorphisms are dense.  To state this more precisely, suppose $M$ is a compact manifold with boundary, and denote by $\Homeo_0(M,\partial M)$ the connected component of the identity in the group of homeomorphisms of $M$ which fix every point of the boundary $\partial M$.  Note that if $M=I$, then $\Homeo_0(M,\partial M)=\Homeo_+(I)$, and $\Homeo_0(M,\partial M)$ is generically topologically $2$-generated by our results.  By contrast, if $M\neq I$, we have the following.

\begin{thm}  Let $M$ be a compact connected manifold with boundary which is not homeomorphic to a closed interval.  Then the set $\Omega_n$ of all $n$-tuples which topologically generate $\Homeo_0(M,\partial M)$ is contained in a closed nowhere dense subset of $(\Homeo_0(M,\partial M))^n$ for all positive integers $n$.
\end{thm}

This theorem is shown using a standard ping-pong argument, which allows one to find open neighborhoods in a group whose elements always generate free discrete subgroups.  The same argument shows that $\Homeo_+(\mathbb{S}^1)$ is not generically $n$-generated for any $n$, nor is $\Homeo_0(M)$ for any compact manifold $M$ without boundary.  This argument does not apply on the interval, which stands alone among the compact connected manifolds in this regard.

 \vspace{1cm}
 
 \section{Finiteness of the Rank of Diff$_{+}^1(I)$}

We begin by showing that the topological commutator subgroup of $\mathrm{Diff}_+^1(I)$ is finitely generated.

\begin{thm} \label{thm:diff(I)}  Let $G=\{f\in\mathrm{Diff}_+^1(I):f'(0)=f'(1)=1\}$, endowed with the subgroup topology from $\Diff_+^1(I)$.  Then $G$ is topologically $7$-generated.
\end{thm}

   {\bf Proof.}  By Lemma \ref{lem:tsuboi}, $G=\overline{[G,G]}=\overline{[\mathrm{Diff}_+^1(I),\mathrm{Diff}_+^1(I)]}$.  Let $$S = \{f\in G \ | \ \mathrm{Fix}(f)\cap (0,1) = \emptyset \}.$$  We observe that $S$ consists of exactly those maps satisfying the hypothesis of Lemma \ref{lem:introroot}, and that $S^2=G$.
 
 \medskip
 
 We will prove that $G$ is generated by seven diffeomorphisms $\phi , \psi , f, g, u, v$ and $h$ by constructing these diffeomorphisms explicitly.
 
 \medskip
 
 Let $I_n = (a_n, b_n), n\geq 0$ be open intervals in $(0,1)$ with mutually disjoint closures such that $(a_n)_{n\geq 0}$ is decreasing and $\displaystyle \lim _{n\to \infty}a_n = 0$. Let also $c_n\in(a_n,b_n), n\geq 0$.  Let $(J_n)_{n\geq0}$ be open mutually disjoint intervals in $(0,1)$ such that $\overline{I_n}\subseteq J_n$.

  \medskip
  
  We choose $\phi , \psi \in G$ such that the following conditions hold:
  
  \medskip
  
  (a-i) $\phi (x) > x, \forall x\in (0,1)$;
  
  \medskip
  
  (a-ii) $\psi (c_0) = c_0$, moreover, $\psi (x) > x, \forall x\in (0, c_0)$ and $\psi (x) < x, \forall x\in (c_0, 1)$; and
  
  \medskip
  
   (a-iii) $\phi ^{-1}(a_n) = a_{n+1}, \phi ^{-1}(b_n) = b_{n+1}, \phi ^{-1}(c_n) = c_{n+1}, \forall n\geq 0$.
	
	   \medskip
   
   \begin{figure}[h!]
  \includegraphics[width=3.5in,height=3.5in]{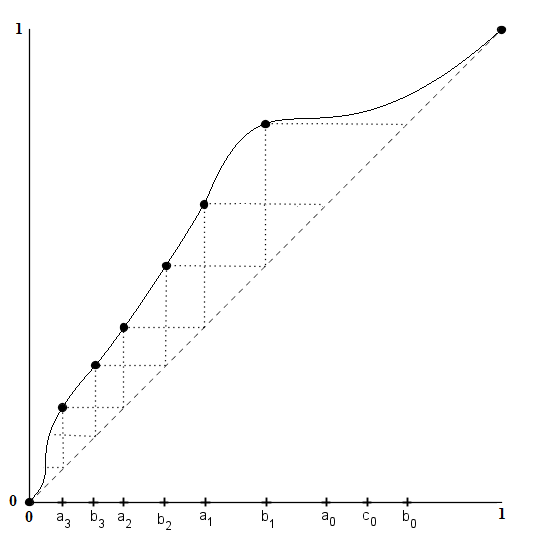}
\caption{the map $\phi $.}
\label{labelname}
\end{figure}

 \begin{figure}[h!]
  \includegraphics[width=3.5in,height=3.5in]{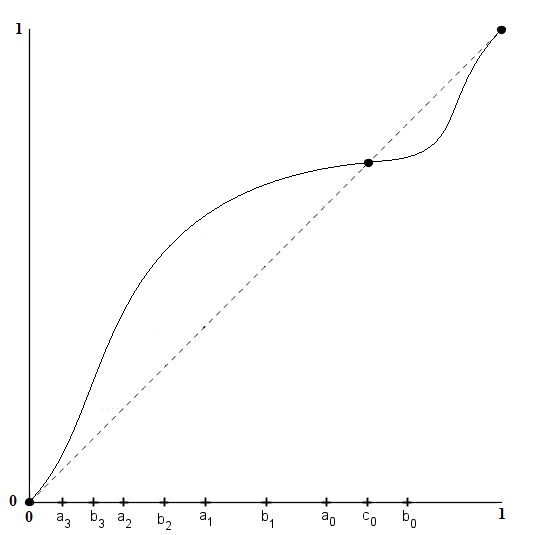}
\caption{the map $\psi $.}
\label{labelname}
\end{figure}
   
	For a natural number $m$, let us denote by $\Phi_n^m$ the map $\phi^n\psi^{-m}\phi^{-n}$.  In the construction that follows, for large values of $m$, the intuitive purpose of $\Phi_n^m$ is to dilate maps (via conjugation) which are supported on $[a_n,b_n]$, to approximate maps which are supported on $[0,1]$.
		
  \medskip
  
  Now we start constructing $f, g, u, v$ and $h$. For this, let $\eta _1, \eta _2, \dots $ be a dense sequence in the subset $\{[\omega _1\omega _2, \omega _3\omega_4] : \omega _i\in S, 1\leq i\leq 4\}$ where we consider the standard $C^1$ metric on $G$.  Let also the sequence be enumerated in such a way that each map $\eta_n$ is listed infinitely many times.  Note that $\overline{\langle \eta_n:n\geq1\rangle}=\overline{[S^2,S^2]}=G$ where the latter equality follows from Lemma \ref{lem:tsuboi}.  Therefore it suffices for our construction to guarantee that each map $\eta_n$ lies in the closure of the group $\langle\phi,\psi,f,g,u,v,h\rangle$.
	
	\medskip
	
	We let $f(x) = g(x) =u(x) = v(x) = x, \forall x\in  I \backslash \displaystyle \mathop{\sqcup }_{n\geq 1}I_n$, and $h(x)=x, \forall x\in I \backslash \displaystyle \mathop{\sqcup }_{n\geq 1}J_n$.  We will complete the definition of $f, g, u, v$ in the intervals $I_n = (a_n, b_n), n\geq 1$, and $h$ in the intervals $J_n$, inductively. 
	Let $(\epsilon _n)_{n\geq 1}$ and $(\delta _n)_{n\geq1}$ be sequences of positive numbers decreasing to zero.  Throughout the proof, if we state that a map $f$ is $\alpha$-close to a map $g$ on an interval $J$, we mean that $f$ and $g$ are uniformly $\alpha$-close in the $C^1$-metric on $J$, i.e. $\displaystyle\sup_{x\in J}|f'(x)-g'(x)|<\alpha$.
  
  \medskip
	
	In the interval $I_1$ we choose $f,g,u,v$ so that:
	
	\medskip
	
	(b-i) $f'(a_1) = g'(a_1) = u'(a_1) = v'(a_1) = 1$ and $f'(b_1) = g'(b_1) = u'(b_1) = v'(b_1) = 1$;
	
	\medskip
	
	(b-ii) $f,g,u,v$ are $\delta_1$-close to identity on the interval $I_1$; and
	
	\medskip
	
	(b-iii) $\Phi_1^{m_1}[f^{M_1}g^{M_1},u^{M_1}v^{M_1}](\Phi_1^{m_1})^{-1}$ is $\epsilon_1$-close to $\eta_1$ on the interval $\Phi_1^{m_1}(I_1)$ for some sufficiently big integers $M_1$ and $m_1$.
	
	\medskip
	
We may achieve condition (b-iii) for sufficiently big $M_1$ via Lemma \ref{lem:introroot}.  We define $h$ so that $h'(x)=1$ if $x$ is an endpoint of $J_1$; $h$ is $\delta_1$-close to identity on $J_1$; and $h(x)>x$ for all $x\in J_1$.  Note that for some sufficiently large integer $p_1$, the supports of $h^{p_1}fh^{-p_1}$ and $h^{p_1}gh^{-p_1}$ in $J_1$ will each be disjoint from the supports of $u$ and $v$ in $J_1$.  It follows that the commutator $[h^{p_1}f^{M_1}g^{M_1}h^{-p_1},u^{M_1}v^{M_1}]$ acts as the identity on $J_1$.
	
  \medskip	
  
  Suppose now the maps $f, g, u$ and $v$ are defined on $\displaystyle \mathop{\sqcup }_{1\leq i\leq n}I_i$ and $h$ is defined on $\displaystyle \mathop{\sqcup}_{1\leq i\leq n}J_i$	so that
  
  \medskip
  
   (c-i) $f'(x) = g'(x) = u'(x) = v'(x) = 1, \forall x\in \{a_1, \dots , a_n\}\cup \{b_1, \dots , b_n\}$;
   
  \medskip
  
	(c-ii) $f,g,u,v,h$ are $\delta_i$-close to identity on $J_i$, $1\leq i\leq n$;
	
	\medskip
	
	(c-iii) the map $\Phi_n^{m_n}[h^{p_{n-1}}f^{M_n}g^{M_n}h^{-p_{n-1}},u^{M_n}v^{M_n}](\Phi_n^{m_n})^{-1}$ is $\epsilon_n$-close to $\eta_n$ on the interval $\Phi_n^{m_n}(I_n)$, for some integers $p_{n-1}, m_n, M_n$; and
	
	\medskip
	
	(c-iv) for some integer $p_n>p_{n-1}$, $\Phi_1^{m_1}[h^{p_n}f^{M_1}g^{M_1}h^{-p_n},u^{M_1}v^{M_1}](\Phi_1^{m_1})^{-1}$ is the identity on $\displaystyle \mathop{\sqcup}_{i\leq i\leq n}I_i$.
   
   \medskip 
   
  Then we define the diffeomorphisms $f, g, u, v$ on the interval $I_{n+1}$ such that
  
  \medskip
  
  (d-i) $f'(x) = g'(x) = u'(x) = v'(x) = 1$ for all $x\in \{a_{n+1}, b_{n+1}\}$;
  
  \medskip
  
  (d-ii) $f, g, u, v$ are $\delta_{n+1}$-close to the identity on $I_{n+1}$; and
  
   \medskip
	
	(d-iii) the map $\Phi_{n+1}^{m_{n+1}}[f^{M_{n+1}}g^{M_{n+1}},u^{M_{n+1}}v^{M_{n+1}}](\Phi_{n+1}^{m_{n+1}})^{-1}$ is $\epsilon_{n+1}$-close to $\eta_{n+1}$ on the interval $\Phi_{n+1}^{m_{n+1}}(I_{n+1})$, for some integers $m_{n+1},M_{n+1}$.
	
	\medskip
	
	We complete the inductive construction by choosing $h$ on $J_{n+1}$ in such a way that $h$ is $\delta_{n+1}$-close to identity on $J_{n+1}$ and also tangent to identity at the endpoints of $J_{n+1}$, and $h(x)>x$ $\forall x\in J_{n+1}$.  We also require $h$ be so close to identity on $J_{n+1}$ that the map $$\Phi_{n+1}^{m_{n+1}}[h^{p_n}f^{M_{n+1}}g^{M_{n+1}}h^{-p_n},u^{M_{n+1}}v^{M_{n+1}}](\Phi_{n+1}^{m_{n+1}})^{-1}$$ remains $\epsilon_{n+1}$-close to $\eta_{n+1}$ on $\Phi_{n+1}^{m_{n+1}}(I_{n+1})$.
	
	\medskip
	
	We note that for some sufficiently large integer $p_{n+1}>p_n$, the supports of $h^{p_{n+1}}fh^{-p_{n+1}}$ and $h^{p_{n+1}}gh^{-p_{n+1}}$ in $J_{n+1}$ will each be disjoint from the supports of $u$ and $v$ in $J_{n+1}$, and thus the commutator $$\Phi_{n+1}^{m_{n+1}}[h^{p_{n+1}}f^{M_{n+1}}g^{M_{n+1}}h^{-p_{n+1}},u^{M_{n+1}}v^{M_{n+1}}](\Phi_{n+1}^{m_{n+1}})^{-1}$$ will act trivially on $J_{n+1}$, preserving the inductive hypothesis. 
   
   
   
   
   \medskip

 By induction, we extend the maps $f, g, u, v, h$ to the whole interval $I$.  Let $X_n=X_n(\phi,\psi,f,g,u,v,h)=[h^{p_{n-1}}f^{M_n}g^{M_n}h^{-p_{n-1}},u^{M_n}v^{M_n}]$.  By construction, the word $\Phi_n^{m_n}X_n(\Phi_n^{m_n})^{-1}$ is:

\begin{enumerate}
		\item $\epsilon_n$-close to $\eta_n$ on the interval $[\Phi_n^{m_n}(a_n),\Phi_n^{m_n}(b_n)]$;
		\item the identity on $[\Phi_n^{m_n}(b_n),1]$; and
		\item $\gamma_n$-close to identity on $[0,\Phi_n^{m_n}(a_n)]$, where $\gamma_n$ depends only on $\delta_n$ and $m_n$, and where $\gamma_n\rightarrow0$ as $\delta_n\rightarrow0$.
\end{enumerate}

By choosing $(m_n)$ large enough we may guarantee that $\eta_n$ is $\epsilon_n$-close to identity on the intervals $[\Phi_n^{m_n}(b_n),1]$ and $[0,\Phi_n^{m_n}(a_n)]$.  Then, by choosing $(\delta_n)$ small enough, we guarantee that $\Phi_n^{m_n}X_n(\Phi_n^{m_n})^{-1}$ is $\epsilon_n$-close to $\eta_n$ on $I$.  Since each $\eta_n$ is listed infinitely many times and $\epsilon_n\rightarrow0$, this shows that $\eta_n$ is in the closure of the group generated by $\phi,\psi,f,g,u,v$ and $h$.  $\square $
   
  \bigskip
   
   We would like to describe the construction of the proof of Theorem \ref{thm:diff(I)} briefly in words. At each of the domains $I_n$ the maps $f, g, u,v,h$ generate some words $X_n$, each of which is a small copy of a map close to $\eta _n$. To dilate this small copy we use the maps $\phi $ and $\psi $. $\phi $ moves the domain $I_n$ to $I_0$ and $\psi ^m$ dilates the domain $I_0$ to the interval containing $(\alpha _n , 1- \alpha _n)$ where $\alpha _n$ tends to zero. The word $X_n(f,g)$ has been realized as a commutator of a word $Z_n(f,g,u,v)$ and the conjugate of a word $Y_n(f,g,u,v)$ by some power of $h$. As a result of this we are able to make an arrangement that at all the intervals $I_i, i < n$ conjugating by the chosen power of $h$ disjoints the supports of $Y_n(f,g,u,v)$ and $Z_n(f,g, u,v)$ while $h$ itself is very close to the identity in the $C^1$ norm. 
	
	\bigskip
   
 {\bf Proof of Theorem \ref{thm:main}.} We have shown that $G$ is topologically generated by $7$ maps $\phi,\psi,f,g,u,v,h$, and now as a consequence we show that $\mathrm{Diff}_+^1(I)$ is topologically $10$-generated.  We note that $G$ is a closed normal subgroup in $\mathrm{Diff}_+^1(I)$, and that the quotient $\mathrm{Diff}_+^1(I)/G$ is topologically isomorphic to the additive group $\mathbb{R}^2$.  Since $\mathbb{R}^2$ is topologically 3-generated, the claim follows from the following simple lemma; for the reader's convenience we present its proof.
 
 \medskip
 
 \begin{lem}\label{lem:rank} Let $H$ be a topological group and $N$ be a normal subgroup. Then $$\mathrm{topological \  rank}(H) \leq \mathrm{topological \ rank}(H/N) + \mathrm{topological \  rank}(N).$$   
 \end{lem} 

 \medskip
 
 {\bf Proof.} For $x\in H, [x]$ will denote the corresponding element in $H/N$. We may and will assume that the groups $H/N$ and $N$ have finite topological ranks. Let $\mathrm{topological \ rank}(H/N) = p, \mathrm{topological \  rank}(N) = q$, and $\langle [g_1], ... ,[g_p]\rangle = H/N, \langle a_1, ... , a_q\rangle = N.$  

 \medskip
 
 We will prove that $\langle g_1, ... , g_p, a_1, ... , a_q \rangle = H$. Let $g\in H$ be an arbitrary element, and $U$ be an arbitrary non-empty open neighborhood of $1\in H$. Consider the open set $V :=UN$ in $H$. Let $V'$ be the projection of $V$ in $H/N$. We can  view $V'$ as an open neighborhood of the identity in $H/N$. Then there exists a word $w([g_1], ... , [g_p])$ such that $$[g]^{-1}w([g_1], ... , [g_p])  \in V'.$$

 \medskip
 
 This means that $g^{-1}w(g_1, ... , g_p) \in V = UN$. Let  $g^{-1}w(g_1, ... , g_p) = fn$, where $f\in U, n\in N$. Then there exists an open neighborhood $A$ of the identity in $N$ such that $fA \subset U$. 

 \medskip
 
 Now, we can find a word $X(a_1, ... , a_q)$ such that $nX(a_1, ... , a_q) \in A$. Then $g^{-1}w(g_1, ... , g_p)X(a_1, ... , a_q) \in U$.

 \medskip
 
Since $U$ is arbitrary, we are done. $\square $ 

\medskip

We remark that the analogue of Lemma \ref{lem:rank} for local rank is also true; see Lemma 4.5 in \cite{GL}.  Thus $\Diff_+^1(I)$ is locally topologically finitely generated.

 \bigskip
 
 \section{Roots of Diffeomorphisms}
 
  In this section we will prove Lemma \ref{lem:introroot} (represented here as Lemma \ref{lem:roots}).  We will need the following.
  
  \medskip
  
  \begin{defn} Let $l\geq 1$. A finite sequence $(t_i)_{1\leq i\leq N}$ of real numbers is called {\em $l$-quasi-monotone} if there exists $1 < i_1 \leq \dots \leq i_s < N$ for some $s\in \{1, \dots ,l\}$ such that each of the subsequences $(t_{i_j}, \dots , t_{i_{j+1}})$ is monotone for every $0\leq j\leq s$, where $i_0 = 1, i_{s+1} = N$. A 1-quasi-monotone sequence will be simply called {\em quasi-monotone}.
  \end{defn}
  
  \medskip  
 
 \begin{lem}\label{lem:roots} Let $f\in \mathrm{Diff}_{+}(I), \ f'(0) = f'(1) = 1$ and $r > 0$. Let also $f$ has no fixed point in $(0,1)$. Then there exists a natural number $N\geq 1$ and $g\in \mathrm{Diff}_{+}(I)$ such that $||g||_1 < r$ and $||g^N-f||_1 < r$. 
 \end{lem}
 
 \medskip
 
 {\bf Proof.} Let $f_1\in B_{r/4}(f)$ be a diffeomorphism of class $C^2$ with no fixed point in $(0,1)$, such that $f_1''$ changes its sign at most $l$ many times, for some finite $l$.
 Let also $0 < x_0 < x_1 < \dots < x_n < 1$ such that $f_1(x_i) = x_{i+1}, 0\leq i\leq n-1$, moreover, $$\max \{x_0, 1-x_n\} < r/4 \ \mathrm{and} \ \displaystyle \mathop{\sup }_{x\in [0, x_1]\cup [x_{n-1},1]}|f_1'(x) - 1|< r/4.$$  
  
 \medskip
 
 Let $N\geq 1$ be sufficiently big with respect to $\max \{n, l\}$, and $z_0, z_1, z_2,$ \ $\dots , z_{nN}$ be a sequence such that $z_{iN} = x_i, 0\leq i\leq n$, $z_{k} = x_0 + \frac{x_1-x_0}{N}k, 0\leq k\leq N$, and $z_{iN+j} = f_1(z_{iN-N+j}), 1\leq i\leq n-1, 1\leq j\leq N$.
 
 \medskip
 
 Now, we let $$\overline{g'}(z_i) = \frac{z_{i+2}-z_{i+1}}{z_{i+1}-z_i} = \frac{g(z_{i+1})-g(z_{i})}{z_{i+1}-z_i}, 0\leq i\leq nN-2$$ and define the diffeomorphism $g\in \mathrm{Diff}_{+}(I)$ as follows. Firstly, we let $g(z_i) = z_{i+1}, 0\leq i\leq nN-1$. Then we let $$g'(z_i) = \frac{z_{i+2}-z_{i+1}}{z_{i+1}-z_i} + \epsilon _i^{(N)}, 0\leq i\leq nN-2$$ where $\epsilon _i^{(N)} < 0$ if $\overline{g'}(z_i) < \overline{g'}(z_{i+1})$, and $\epsilon _i^{(N)} > 0$ if $\overline{g'}(z_i) > \overline{g'}(z_{i+1})$; also, in case of the equality $\overline{g'}(z_i) = \overline{g'}(z_{i+1})$, we let $\epsilon _i^{(N)} = 0$.  
 
 \medskip
 
 Let us observe that we already have $g^N(z_i) = z_{i+N}, 0\leq i\leq nN-N$. Then $|g^N(z_i) - f_1(z_i)| < r/4$ for all $1\leq i\leq nN-1$ hence we have $||g^N(x) - f_1(x)||_0 < r/2$ {\bf (1)} \ for all $x\in [0,1]$ (i.e. for all possible extensions of $g$). On the other hand, by Mean Value Theorem, applied to the iterates $f_1^s, 1\leq s\leq n$ and their derivatives $(f_1^s)', 1\leq s\leq n$, for sufficiently big $N$ and sufficiently small $\epsilon ^{(N)}: = \displaystyle \mathop{\max }_{1\leq i\leq N}|\epsilon _i^{(N)}|$, we obtain that $$|g'(z_i)-1| < r/4, 0\leq i\leq nN-2.$$
 
 \medskip
 
 By chain rule, we also have $$(g^N)'(z_i) = \frac{z_{i+N+2}-z_{i+N+1}}{z_{i+1}-z_i} + \delta _i^{(N)}, 0\leq i\leq nN-N$$ where $\displaystyle \mathop{\max }_{1\leq i\leq N}|\delta _i^{(N)}|\to 0$ as $N\to \infty $ and $\epsilon ^{(N)}\to 0$. Then by Mean Value Theorem, if $N$ is sufficiently big, we obtain $$|(g^N)'(z_i) - f_1'(z_i)| < r/4, 0\leq i\leq nN-n.$$

 \medskip
 
 Now we need to extend the definition of $g'(x)$ to all $x\in [0,1]$ such that we have $|(g^N)'(x) - f_1'(x)| < r/2$ {\bf (2)} for all $x\in [0,1]$. For all $x\in (z_i, z_{i+1}), 0\leq i\leq nN-1$, we define $g'(x)$ such that $g'(x)$ is monotone on the interval $(z_i, z_{i+1})$.  (A monotone extension is possible because of the $\epsilon_i^N$ error terms.)
 
 \medskip
 
 Then we have $|g'(x)-1| < r/4$ for all $x\in (x_0, x_{n})$. By uniform continuity of the derivative $f_1'(x)$ and by Mean Value Theorem, if $N$ is sufficiently big, we also obtain the inequality {\bf (2)}. Indeed, it suffices to show that $|(g^N)'(z) - (g^N)'(z_i)| < r/4$ for all $z\in (z_i, z_{i+1}), 0\leq i\leq nN-n.$
 
 \medskip
 
 By Mean Value Theorem, for sufficiently big $N$ the sequence $$(g'(z_i), g'(z_{i+1}), \dots , g'(z_{i+N}))$$ is $l$-quasi-monotone. For simplicity, let us assume that this sequence is quasi-monotone. 
 
 \medskip
 
 Without loss of generality, let us assume that $g'(z_i) < g'(z_{i+1})$. Then for some $p\in \{1, \dots , N\}$ we have $g'(z_i) \leq g'(z_{i+1}) \leq \dots \leq g'(z_{i+p})$ and $g'(z_{i+p}) \geq \dots \geq g'(z_{i+N})$. 
 
 \medskip
 
 Then $$(g^N)'(z) \geq g'(z_i)\dots g'(z_{i+p-1})g'(z_{i+p+1})\dots g'(z_{i+N}) = g'(z_i)\dots g'(z_{i+N-1})\frac{g'(z_{i+N})}{g'(z_{i+p})}.$$ For sufficiently big $N$, we obtain that $(g^N)'(z) \geq (g^N)'(z_i) - r/4$. Similarly, we also have the inequality $(g^N)'(z) \leq (g^N)'(z_i) + r/4$. Thus we obtain the inequality {\bf (2)}.
 
 \medskip
 
 If the sequence $(g'(z_i), g'(z_{i+1}), \dots , g'(z_{i+N}))$ is $l$-quasi-monotone (i.e. in the most general case) then, similarly, we will obtain $$(g^N)'(z) \geq g'(z_i)\dots g'(z_{i+N-1})\displaystyle \mathop{\Pi }_{1\leq j\leq s}\frac{g'(z_{p_j})}{g'(z_{q_j})}$$ where $s\leq l$ and $p_1, q_1, \dots , p_s, q_s\in \{i, i+1, \dots , i+N\}$. Then, by Mean Value Theorem, for sufficiently big $N$, we obtain the inequality $(g^N)'(z) \geq (g^N)'(z_i) - r/4$, and similarly, the inequality $(g^N)'(z) \leq (g^N)'(z_i) + r/4$.  
 
 \medskip

  The inequalities {\bf (1)} and {\bf (2)} imply the claim. $\square $
 
 \bigskip

 At the end of this section we would like to quote the following result of T.Tsuboi which is one of the many perfectness results of various homeomorphism and diffeomorphism groups. We have used this result in Section 2.

 \begin{lem}\label{lem:tsuboi} The group $G=\{f\in\mathrm{Diff}_+^1(I):f'(0)=f'(1)=1\}$ is perfect.
 \end{lem}
 
 \medskip
 
 {\bf Proof.} See Theorem 4.1 in \cite{T}.  
 
 \vspace{1cm}
 
 \section{Generic Subgroups}
 
 In this section we will prove that a generic 2-generated subgroup of $\mathrm{Homeo}_+(I)$ is dense.  Moreover, we observe that $I$ is the unique compact connected manifold whose homeomorphism group exhibits this property, for any size of finite generating set.  That is, we prove that for any compact connected manifold $M$ other than $I$ and for any $n\geq1$, a generic $n$-generated subgroup is not dense.  In fact, we will exhibit an open set such that a subgroup generated by an $n$-tuple from this open set (in the product $\mathrm{Homeo}_0(M)^n)$ is even discrete.  If $M$ is a compact manifold with boundary, then there is even an open dense set of generators in the product group $\mathrm{Homeo}_0(M)^n$ which generate a discrete subgroup.  So the behavior of generic subgroups in a compact manifold $M$ with boundary of dimension greater than or equal to $2$ is opposite from the case of $I$.

 \begin{thm}  For a comeager set of pairs $(f,g)\in(\mathrm{Homeo}_+(I))^2$, the subgroup $\Gamma=\langle f,g\rangle$ is dense in $\mathrm{Homeo}_+(I)$.
 \end{thm}

{\bf Proof.}  Let $A$ be the set of all pairs $(f,g)$ which generate a dense subgroup of $\mathrm{Homeo}_+(I)$.  We will show $A$ is a dense $G_\delta$ set in $\mathrm{Homeo}_+(I)$.  Let $\rho$ denote the standard uniform metric on $\mathrm{Homeo}_+(I)$.  If $D\subseteq\mathrm{Homeo}_+(I)$ denotes a countable dense subset of $\mathrm{Homeo}_+(I)$, and $W$ denotes the set of all words in a free group on two generators, then we have\\

\medskip

\begin{center} $(f,g)\in A~\leftrightarrow~\forall d\in D~\forall n\in\mathbb{Z}^+~\exists w\in W~ \rho(w(f,g),d)<\frac{1}{n}$.
\end{center}
\vspace{.3cm}

Since the predicate above is an open condition (by continuity of the word maps $w$), this equivalence shows $A$ is $G_\delta$.  So it remains for us to show that $A$ is dense.  Let $(f,g)$ be an arbitrary pair of maps in $\mathrm{Homeo}_+(I)$ and let $\delta>0$ be arbitrary; we will show there is a pair of maps $(\tilde{f},\tilde{g})$ such that $\rho(f,\tilde{f})<\delta$, $\rho(g,\tilde{g})<\delta$, and $\Gamma=\langle\tilde{f},\tilde{g}\rangle$ is dense.

\medskip

Firstly, it is clear that the set of all pairs $(f,g)$ which do not share a common fixed point in $(0,1)$ is dense in $(\mathrm{Homeo}_+(I))^2$.  This means we can assume without loss of generality that $f$ and $g$ do not fix a common point in $(0,1)$, and we do so now.

\medskip

We construct the maps $\tilde{f}$, $\tilde{g}$ as follows.  Since $f$ and $g$ fix the endpoints $0$ and $1$ of $I$, we may find a distance $\alpha>0$ so small that $f(x),g(x)<\delta$ whenever $x<\alpha$.  Let $y_0\in(0,\alpha)$ and $\tilde{g}$ be any interval homeomorphism such that the following conditions are satisfied:\\

\medskip

\begin{enumerate}
		\item $y_0$ is the least non-zero fixed point of $\tilde{g}$;
		\item $\tilde{g}(x)>x$ for all $x\in(0,y_0)$; and
		\item $\tilde{g}$ agrees with $g$ on $[\alpha,1]$.
\end{enumerate}
\vspace{.3cm}

Note that since we have only modified $g$ on the interval $[0,\alpha)$ to produce $\tilde{g}$, we have $\rho(g,\tilde{g})<\delta$ as desired.

\medskip

Now let $x_0\in(0,y_0)$ be arbitrary.  For each $n\geq1$ set $x_n=\tilde{g}^{-n}(x_0)$, so the sequence $(x_n)$ decreases strictly and converges to $0$.  Let $\phi$ and $\psi$ be any two elements of $\mathrm{Homeo}_+([x_1,x_0])$ which generate a dense subgroup of $\mathrm{Homeo}_+([x_1,x_0])$.  Construct $\tilde{f}$ to satisfy the following conditions:\\

\begin{enumerate}
		\item $\tilde{f}$ fixes each point in the sequence $(x_n)$;
		\item $\tilde{f}(x)>x$ for all $x\in(x_0,y_0]$;
		\item $\tilde{f}$ agrees with $\phi$ on $[x_1,x_0]$ and with $\tilde{g}^{-1}\psi\tilde{g}$ on $[x_2,x_1]$
		\item $\tilde{f}$ agrees with $\tilde{g}^{-2}\tilde{f}\tilde{g}^2$ on $[x_{n+1},x_n]$ for each $n\geq 2$;
		\item $\tilde{f}$ agrees with $f$ on $[\alpha,1]$; and
		\item $\tilde{f}$ has no common fixed point with $\tilde{g}$ on $(y_0,\alpha)$.
\end{enumerate}
\vspace{.3cm}

Once again we have $\rho(f,\tilde{f})<\delta$ since we have only modified $f$ on $[0,\alpha)$.  Note that $\tilde{f}$ is constructed in such a way that the maps $\tilde{f}|_{[x_{n+1},x_n]}$ and $\tilde{g}\tilde{f}\tilde{g}^{-1}|_{[x_{n+1},x_n]}$ generate a dense subgroup of $\mathrm{Homeo}_+([x_{n+1},x_n]$ for each $n\geq 0$.

\medskip

We claim that $\tilde{f}$ and $\tilde{g}$ generate a dense subgroup $\Gamma=\langle\tilde{f},\tilde{g}\rangle$ in $\mathrm{Homeo}_+(I)$.  To see this, let $\varphi\in\mathrm{Homeo}_+(I)$ and $\epsilon>0$ be arbitrary; we will produce a word in $\tilde{f}$ and $\tilde{g}$ which is uniformly $\epsilon$-close to $\varphi$.

\medskip

Let $\gamma>0$ be so small that $\varphi(x)<\epsilon$ whenever $x<\gamma$ and $\varphi(x)>1-\epsilon$ whenever $x>1-\gamma$.  Note that the perturbed maps $\tilde{f}$ and $\tilde{g}$ still do not share a common fixed point in $(0,1)$.  This means we can find a map $h\in\Gamma$ for which $h(y_0)>1-\gamma$.  Set $F=h\tilde{f}h^{-1}$ and $G=h\tilde{g}h^{-1}$.  So $F,G\in\Gamma$, and $\mathrm{Fix}(F)=h(\mathrm{Fix}(\tilde{f}))$ and $\mathrm{Fix}(G)=h(\mathrm{Fix}(\tilde{g}))$.

\medskip

Since $(h(x_n))\rightarrow0$, let $N$ be an integer so large that $h(x_{N+1})<\gamma$.  Since $F$ has no fixed points on the interval $(h(x_0),h(y_0)]$, and the point $G(h(x_0))$ lies on this interval, we may find an integer $M$ so large that $F^M(G(h(x_0)))>h(y_0)$.  Set $\Phi=G^{-(N+1)}F^MG^{N+1}\in\Gamma$.  Then $\Phi$ fixes $h(x_{N+1})$, and maps $h(x_N)$ to a point strictly greater than $h(y_0)$.  Defining $a$ and $b$ to be these image points respectively, we have

\medskip

 $$a=\Phi(h(x_{N+1}))=h(x_{N+1})<\gamma<1-\gamma<h(y_0)<\Phi(h(x_{N+1}))=b.$$

\bigskip

Now since the maps $\tilde{f}|_{[x_{N+1},x_N]}$ and $\tilde{g}\tilde{f}\tilde{g}^{-1}|_{[x_{N+1},x_N]}$ generate a dense subgroup of $\mathrm{Homeo}_+([x_{N+1},x_N])$, it follows that the maps $F|_{[h(x_{N+1}),h(x_N)]}$ and $GFG^{-1}|_{[h(x_{N+1}),h(x_N)]}$ generate a dense subgroup of $\mathrm{Homeo}_+([h(x_{N+1}),h(x_N)]$.  In turn, we have that $\Phi F \Phi^{-1}|_{[a,b]}$ and $\Phi GFG^{-1}\Phi^{-1}|_{[a,b]}$ generate a dense subgroup of $\mathrm{Homeo}_+([a,b])$.

\medskip

Now note that $|\varphi(a)-a|<\max(\varphi(a),a)<\epsilon$ since $a<\gamma$, and $|\varphi(b)-b|<\max(1-\varphi(b),1-b)<\epsilon$ since $b>1-\gamma$.  This $\epsilon$-closeness at the endpoints $a$ and $b$ ensures that we may find a word $w=w(\Phi F\Phi^{-1},\Phi GFG^{-1}\Phi^{-1})\in\Gamma$ which is uniformly $\epsilon$-close to $\Phi$ on the interval $[a,b]$.  Since this word $w$ must fix $a$ and $b$ and $a<\gamma<1-\gamma<b$, it follows again from our choice of $\gamma$ that $w$ is uniformly $\epsilon$-close to $\varphi$ on $[0,a]$ and $[b,1]$ as well.  So $\rho(w,\varphi)<\epsilon$ and $\Gamma$ is dense, as claimed. $\square $

\medskip

 \begin{thm} \label{thm:ping-pong} Let $k\geq 0$ and $M$ be a compact connected manifold which is not homeomorphic to a closed interval. Then for every $n\geq 2$, there exists $f_1, f_2, \dots, f_n \in \mathrm{Homeo}_0(M)$ and $r > 0$ such that for all $\tilde {f}_1, \tilde{f}_2, \dots, \tilde{f}_n$ from the $r$-neighborhood of $f_1, f_2,\dots,f_n$ in $C^0$ metric, the maps $\tilde {f}_1, \tilde {f}_2,\dots\tilde{f}_n$ generate a discrete free subgroup of rank $n$ in $C^0$ metric. Moreover, if $M$ is a compact connected smooth manifold, then the same results hold for the groups $\mathrm{Diff}_0^k(M)$ in the $C^k$ metric, $1\leq k\leq \infty$, and we can choose $f_1,f_2,\dots,f_n$ from $\mathrm{Diff}_0^{\infty }(M)$.
 \end{thm}
 
 \medskip
 
 {\bf Proof.} We will use a standard ping-pong argument. Let $U$ be an open chart of the $M$ homeomorphic to an open unit ball in $\mathbb{R}^d$ where $d = \mathrm{dim} M$. Let also $A_1, A_2, \dots, A_n, B_1, B_2, \dots, B_n$ be open balls in $U$ with mutually disjoint closures such that $\displaystyle\bigcup_{i=1}^n\overline {A_i\cup B_i} \subset U$. 
 
 \medskip
 
 We can find homeomorphisms (or $C^{\infty }$-diffeomorphisms if $M$ is a smooth manifold) $f_1,\dots,f_n$ of $M$ fixing $M\backslash U$ such that for all $1\leq i\leq n$,
 
 \medskip
 
 (i) $f_i\left((\displaystyle\bigcup_{\{j:1\leq j\leq n, j\neq i\}} A_j\cup B_j) \cup \overline {A_i}\right) \subset A_i$, and
 
 \medskip
 
 (ii) $f_i^{-1}\left((\displaystyle\bigcup_{\{j:1\leq j\leq n, j\neq i\}} A_j\cup B_j) \cup \overline {B_i}\right) \subset B_i$.
 
 \medskip
 
 
 

  
  Then, for all non-zero integers $m$, for each $i,j\in\{1,\dots, n\}$, $i\neq j$, we have $f_i^m (A_j\cup B_j) \subset A_i\cup B_i$. Let $\epsilon $ be the shortest distance among the balls $A_1, \dots, A_n, B_1, \dots, B_n$, and $\delta > 0$ be such that $\phi B_{\delta }(1)\phi ^{-1}\subset B_{\epsilon }(1)$ for all $\phi \in \{f_1^{\pm 1}, \dots, f_n^{\pm 1}\}$ (here, $B_c(1)$ denotes the ball of radius $c$ around the identity).
  
  \medskip
  
  Then we can claim that for every non-trivial reduced word $W$, we have $||W(f_1,\dots , f_n)||_0 > \delta $. Hence $||W(f_1,\dots , f_n)||_k > \delta $. \ $\square $
 
\medskip

We remark here that although the previous theorem implies that $\mathrm{Homeo}_+(\mathbb{S}^1)$ is not generically topologically $2$-generated, it does have dense $2$-generated subgroups, and we provide a brief proof below.

 \bigskip
  
 \begin{prop} $\Homeo_+(\mathbb{S}^1)$ is topologically $2$-generated.
 \end{prop}
 
 \medskip
 
 {\bf Proof.} Regard $\mathbb{S}^1$ as the interval $[0,1]$ with the endpoints $0$ and $1$ identified.  Let $\phi:[0,1]\rightarrow[0,1/2]$ be any orientation-preserving homeomorphism.  Let $x_0$ and $x_1$ denote the usual two generators of the usual representation of Thompson's group $F$ inside $\Homeo_+(I)$.

\medskip

Let $f=\phi x_1\phi^{-1}$ and $h=\phi x_1^{-1}x_0 \phi^{-1}$, so $f$ and $h$ generate a dense subgroup of $\Homeo_+([0,1/2])$.  Extend $f$ and $h$ to circle homeomorphisms by letting them act trivially on $[1/2,1]$.  Note that the map $f$ is supported on $[1/4,1/2]$, and $f(x)<x$ for all $x\in(1/4,1/2)$, while the map $h$ is supported on $[0,3/8]$, with $h(x)<x$ for all $x\in(0,3/8)$.  It follows that $f$ and $h$ are conjugate in $\Homeo_+(\mathbb{S}^1)$, by a fixed-point free homeomorphism $g$ satisfying $g(0)=1/4$ and $g(3/8)=1/2$.  Let us also choose $g$ in such a way that $3/8$ is not a periodic point of $g$.

\medskip

Let $\Gamma=\langle f,g\rangle$; we claim $\Gamma$ is dense in $\Homeo_+(\mathbb{S}^1)$.  For $n\in\mathbb{N}$, set $a_n=g^n(0)$ and $b_n=g^n(3/8)$.  Note that $a_{n+1}$ lies between $a_n$ and $b_n$ for each $n$, so the intervals $(a_n,b_n)$ and $(a_{n+1},b_{n+1})$ are overlapping.  Since $g$ is fixed-point free and $b_0$ is not periodic, there exists an $N$ for which $b_N$ lies between $a_0$ and $b_0$.  Then the intervals $\{(a_n,b_n)\}_{n=0}^N$ form an open cover of $\mathbb{S}^1$.

\medskip

If $k\in\Homeo_+(\mathbb{S}^1)$ is arbitrary, then we may factorize $k$ as $k=k_0k_1k_2...k_N$, where the support of each $k_n$ is a subset of $[a_n,b_n]$.  Now by construction, $\Gamma$ contains $\langle f,g^{-1}fg\rangle=\langle f,h\rangle$ which is dense in $\Homeo_+[0,1/2]$.  So $k_0$ may be arbitrarily closely approximated by an element $\gamma_0$ of $\Gamma$, since $\supp(k_0)\subseteq[a_0,b_0]\subseteq [0,1/2]$.  Likewise, for each $n\leq N$, we have that $\Gamma=g^n\Gamma g^{-n}$ is dense in $\Homeo_+[g^n(0),g^n(1/2)]$, and $\supp(k_n)\subseteq[a_n,b_n]\subseteq [g^n(0),g^n(1/2)]$.  So $k_n$ may be arbitrarily closely approximated by $\gamma_n\in\Gamma$.  It follows that $k$ is arbitrarily closely approximated by a product of the form $\gamma_0\gamma_1...\gamma_N\in\Gamma$.  Since $k$ is an arbitrary element of $\Homeo_+(\mathbb{S}^1)$ to begin with, we see that $\overline{\Gamma}=\Homeo_+(\mathbb{S}^1)$, i.e. $\Homeo_+(\mathbb{S}^1)$ is topologically generated by $f$ and $g$.

 \bigskip
 
 \begin{thm}  Let $M$ be a compact manifold with boundary such that no connected component of $M$ is homeomorphic to a closed interval.  Let $\Homeo_0(M,\partial M)$ denote the group of homeomorphisms of $M$ which preserve the boundary $\partial M$.  Then the set\\

\begin{center} $\{(f_1,f_2,\dots,f_n)\in(\Homeo_0(M,\partial M))^n:\langle f_1,f_2,\dots,f_n\rangle$ is discrete$\}$
\end{center}
\vspace{.3cm}

\noindent contains an open dense subset of $(\Homeo_0(M,\partial M))^n$.
\end{thm}

{\bf Proof.}  Note that by hypothesis, $\mathrm{dim} M\geq 2$.  Let $(g_1,g_2,\dots,g_n)\in(\Homeo_0(M,\partial M))^n$ and $\epsilon>0$ be arbitrary.  Let $x\in\partial M$ be a boundary point, and let $\delta>0$ be so small that $|y-x|<\delta$ implies $|g_i(y)-x|<\epsilon$ for all $y\in M$, for each $1\leq i\leq n$.  Let $U$ be an open subset in the interior of $M$ so that the closure of $U$ is a subset of the $\delta$-ball about $x$.  Now since $M$ is at least a $2$-dimensional manifold, we can perturb $g_1,\dots,g_n$ to obtain new maps $f_1,\dots,f_n$, which agree with $g_1,\dots, g_n$ respectively outside of the $\delta$-ball about $x$; but which fix the boundary of $U$, and which exhibit the ping-pong behavior described in Theorem  \ref{thm:ping-pong}.  Then $f_1,\dots,f_n$ are $\epsilon$-close to $g_1,\dots,g_n$ respectively, and any sufficiently small perturbations of the new maps $f_1,\dots,f_n$ will generate a discrete subgroup of $\Homeo_0(M,\partial M)$. \ $\square$

 \end{document}